\documentclass[a4paper,10pt, english]{amsart}

\usepackage[latin9]{inputenc}
\usepackage{amsthm}
\usepackage{amsmath}
\usepackage{amssymb}
\usepackage{setspace}
\usepackage{fancyhdr}
\usepackage{txfonts}
\usepackage{enumerate}
\usepackage{babel}
\usepackage{times}
\usepackage[left=4cm,tmargin=3.3cm,textwidth=14cm,textheight=22.65cm]{geometry}
\usepackage{natbib}
\newtheorem{theorem}{Theorem}

\newtheorem*{definition*}{Definition}
\newbox\gnBoxA
\newdimen\gnCornerHgt
\setbox\gnBoxA=\hbox{$\ulcorner$}
\global\gnCornerHgt=\ht\gnBoxA
\newdimen\gnArgHgt

\def\Godelnum #1{%
	\setbox\gnBoxA=\hbox{$#1$}%
	\gnArgHgt=\ht\gnBoxA%
	\ifnum \gnArgHgt<\gnCornerHgt
		\gnArgHgt=0pt%
	\else
		\advance \gnArgHgt by -\gnCornerHgt%
	\fi
	\raise\gnArgHgt\hbox{$\ulcorner$} \box\gnBoxA %
		\raise\gnArgHgt\hbox{$\urcorner$}}
\title{We Belong Together? A plea for modesty in modal plural logic}
\author{Simon Thomas Hewitt}
\AtEndDocument{\bigskip{\footnotesize%
  \textsc{School of Religion, History and the Philosophy of Science, University of Leeds, LS2 9JT. UK.} \par  
  \textit{E-mail address}, S. ~Hewitt: \texttt{s.hewitt@leeds.ac.uk} 
}}
\begin{document}
\maketitle

`We belong together', a doey eyed lover says to his beloved. She, being both better educated in a certain sort of metaphysics than her suitor and more romantic, responds `Yes. Not only in the actual world, but in every possible world. There is no possible world in which we are not together'. Whether this is genuinely touching or simply emetic is a fine judgement. Perhaps more remarkably this expression of lovestruck hyperbole bears a striking similarity to a principle enshrined as current orthodoxy in discussions of the modal logic of plurals. Let $xx$ be some things; then, goes one half of the claim, if $x$ is one of $xx$ then necessarily: if $xx$ exist, then $x$ is one of $xx$. Similarly, if $x$ is not one of $xx$ then necessarily: if $xx$ exist, then $x$ is not one of $xx$. The path from love to logic is a short one indeed.\\

Why should this be of any more than esoteric interest? The principle, which sometimes get calls \textit{plural rigidity}\footnote{I will object to this usage below.} is put to work to important and controversial metaphysical ends. Crucially, Williamson deploys it in arguing both for necessitism (the doctrine that ontology is modally invariant) and for a property-based interpretation of second-order quantification (since he thinks the modal behaviour of plurals rules out a Boolos style plural interpretation) (\cite[]{williamson:13}) (\cite{williamson:10}) (\cite{boolos:98b}). Linnebo develops a foundational programme for set-theory in a modal plural logic strengthened by the addition of statements equivalent to the principle, and whilst he himself proposes a non-standard interpretation of the modalities, its acceptablity with respect to metaphysical modality would provide a fall-back position for someone sympathetic to the approach (\cite{linnebo:13}). Relatedly, so-called rigidity would prevent a development of Hellman's modal structuralism about mathematics using a plural interpretation of second-order logic (\cite{hellman:89}).\\

A number of current philosophical projects, then, turn on this particular claim about the modal behaviour of pluralities\footnote{Here, and throughout, `plurality' is a convenient singularising locution, as is standard in the literature: a plurality is \textit{some things}.}. Yet it is often adopted swiftly without the sustained argument which might reasonably be thought to be required to sustain a claim with such wide-ranging consequences. Recently, Linnebo has made in the direction of making good this deficiency (\cite{linnebo:16}). The present paper argues that Linnebo fails to supply sufficient reasons to adopt the claim at issue, and that in the absence of alternative arguments a question mark must hang over standard appeals to strong modal logics of plurals by philosophers. Of necessity, given the current state of the literature and the importance of Linnebo's work, much of what follows consists of a close criticism of Linnebo, but I go on to suggest that progress in the debate requires us to get clear about the relationship between language and non-linguistic reality and to pay close attention to linguistic use. These thoughts look likely to generalise across the philosophy of modal logic and to call into question much of the present fashion for modalising.\\

I will assume a working knowledge of customary notation and formation rules for plural logic. Good introductions are available in the form of \cite{linnebo:12} and \cite{rayo:07}. A brief overview is provided here in an appendix. Whilst it is standard to offer informal interpretations of modal languages by means of possible-worlds, as the appendix here does for the sake of brevity, the metaphysical associations of this are unhelpful in the present context and present the danger of building in a bias towards semantic realism that sits uncomfortably with suggestions I want to make. I will therefore deploy the less-worn language of truth \textit{with respect to possibilities} or else simply interpret modal operators homophonically. After all, if we do not have a working understanding of expressions such as `necessarily' and `possibly' what do we take modal logics to be interpreting\footnote{I take `modal logic' here to have the sense customary in philosophy, where interpretation is restricted to alethic, and in particular often metaphysical, modalities. Of course, considered as mathematical formalisms, modal logics are interesting in themselves and have a wide variety of applications. For a survey and introduction see \cite{blackburn:02}. Regarding interpretation, I find the suggestion that Kripke semantics -- undoubtedly useful for simplifying proof technique, through construction of counter-models and the like -- shed light on hitherto unanswered questions about the \textit{meaning} of modal operators quite incredible. Nonetheless, my impression is that it is part of philosophical folklore. It is made explicit at \cite[121]{priest:15}.}?\\

The paper proceeds as follows. \S1 lays out the issue to be discussed, and undermines initial motivation for it. \S\S2-4 work through Linnebo's formal arguments, diagnosing a common justificatory circularity. \S5 situates the sceptical position thus arrived at methodologically. I urge caution in the invocation of such a strong modal claim for philosophical argument in the absence of more persuasive arguments for its truth. I suggest that more attention needs to be paid to the use of language in making assertitions and performing deductions as a key, and perhaps the only, constraint on deciding the acceptability of candidate logical principles. This position looks likely to have implications beyond the present case, and to call into question standardly accepted modal axioms.

\section{}

What Linnebo terms \textit{plural rigidity}, we will call \textit{plural necessitism}. The expression is still not perfect (it might suggest a deeper association with Williamson's metaphysics than can be supported), but at least avoids the danger of use-mention confusion that attaches to using the word `rigidity' of a principle stated in the material, rather than formal, mode\footnote{Linnebo himself is keenly aware of this danger (\cite[7]{linnebo:16}).}. Plural necessitism is equivalent to the conjunction of:
\begin{equation}
\tag{NecInc} x\prec yy\rightarrow\Box(Ex\wedge Eyy\rightarrow x\prec yy) \label{NecInc}
\end{equation}

\begin{equation}
 \tag{NecNInc}x\nprec yy\rightarrow\Box(Ex\wedge Eyy\rightarrow x\nprec yy) \label{NecNInc}
\end{equation}

\noindent Here `$E$' is to be read as an existence predicate\footnote{Strictly speaking, we may need to be more careful. We have not formally specified a lexicon, and in particular have not settled the question whether predicates may be untyped, in the sense that they may have argument places for which either singular or plural terms are eligible. From the perspective of tractable formation rules, proof theory and model theory, it is likely to be easier to type predicates. If the lexicon is typed, then `$E$' is doing illegitimate double work in \eqref{NecInc} and \eqref{NecNInc} as formulated. We can get round this straightforwardly, by using superscripts to signify the type of existence predicates in the formal specification of the lexicon, and then omit superscripts in practice as a harmless convention.}, allowing us to work within a variable domain framework. It is routine to modify \eqref{NecInc} and \eqref{NecNInc} for a fixed domain framework simply by removing the conjunction sign and the second conjunct from the antecedent of the second conditional of each formula. In what follows, versions of the principles with and without existence predicates will be denoted by `\eqref{NecInc}' and `\eqref{NecNInc}', with context serving to disambiguate.\\

As we have seen, plural necessitism plays a non-negligible part in motivating or supporting metaphysically contentious positions. What reason, then is there to believe that it is true? It is provably independent of any normal modalisation of \textbf{PFO+} (\cite{hewitt:12}). Motivation will therefore have to take the form of formal argument from independently compelling premises or of informal arguments.\\

As advertised above, our engagement will be mainly with the formal arguments Linnebo offers in support of plural necessitism (hereafter \textit{pn}). Before that, some comment should be made on a case for the plausibility of \textit{pn} which Linnebo develops by analogy with an argument for the modal invariance of set membership, based on extensionality and the necessity of identity\footnote{The case for the necessity of identity invoked in that developed by Barcan Marcus in (\cite{marcus:61}).}. This proceeds from the principle:

\begin{equation}
 \tag{Indisc} xx\equiv yy\rightarrow(\phi(xx)\leftrightarrow\phi(yy)) \label{Ind}
\end{equation}

The subformula `$xx\equiv yy$' abbreviates `$\forall u(u\prec xx\leftrightarrow u\prec yy)$'. This relieves us from the need to read `$\equiv$' as denoting identity on pluralities. This is dialectically useful, since the extent of the similarity of behaviour (in modal contexts) between the referrents of singular and of plural terms is what is at issue. From \eqref{Ind} is derivable\footnote{\textit{Proof sketch}: (1) $xx\equiv yy$ (Ass for CP), (2) $\phi(xx)\equiv\phi(yy)$ ((1, Indisc)), (3) $(xx\equiv xx)\leftrightarrow\Box(xx\equiv xx)$ (Thm), (4) $(xx\equiv yy)\leftrightarrow\Box(yy\equiv xx)$ (3, 2), (5) $\Box(yy\equiv xx)$ (1, 4, MP), (6) $xx\equiv yy\rightarrow\Box(xx\equiv yy)$ (1-5, CP).} :

\begin{equation}
 \tag{Cov} xx\equiv yy\rightarrow\Box(xx\equiv yy) \label{Cov}
\end{equation}

Now \eqref{Cov} does not entail \textit{pn}, since it is compatible with pluralities `drifting' modally. This is easiest to understand using a possible worlds heuristic. Suppose that, at some world $w_{1}$, $xx$ just are $yy$. In order for \eqref{Cov} to be satisfied, we require that at any world accessible from $w_{1}$, $xx$ just are $yy$. But now consider $w_{2}$, accessible from $w_{1}$, such that $xx$ and $yy$ both consist of the members of those pluralities at $w_{1}$ with the addition of $a$, which does not exist at $w_{1}$. We have here no counter-example to \eqref{Cov}, but \textit{pn} fails\footnote{Compare the counter-model to (NecInc) provided (within a fixed domain semantics) in (\cite{hewitt:12}).}.\\

Linnebo takes this drift to be indicative of \textit{soft extensionalism} about pluralities. This is, staying with the possible worlds talk, the doctrine that extensionality holds within worlds but not across them. He argues that this is an unstable position, for the same reason that he holds the parallel position regarding sets to be unsuccessful. With respect to the latter, he draws a comparison between sets and \textit{groups}, where a group is understood as a social entity, such as the Supreme Court justices or the Celtic first eleven:\\
\begin{quote}
 The only reason to accept the principle of extensionality [for sets]$\ldots$ is that a set, unlike a group, is fully specified by its elements. Thus, when tracking a set across possible worlds, there is nothing other than the elements to go on. This ensures that the tracking is rigid. By contrast, when tracking a group, there \textit{is} more than the members to go on$\ldots$ These considerations give rise to a dilemma that applies not only to sets but to any other notion of collection: either we have to give up the principle of extensionality, or else we have to accept the rigidity principle. There is no stable middle ground. Kripke famously taught us that there can be no `soft identity theory' in the philosophy of mind, according to which mental states are identical with physical ones, but only contingently so, only the `hard identity theory' committed to necessary identity. Our present conclusion is analogous. There can be no `soft extensionalism' concerning sets of other kinds of collection, only `hard extensionalism' that incorporates the rigidity claims and the idea of transworld extensionality that they embody. (\cite[6]{linnebo:16}) 
\end{quote}

This is surely correct as regards sets: as Boolos once put the point, if anything deserves to be called analytic of our concept of set it is the axiom of extensionality (\cite[]{boolos:98f}), and the move from this recognition to the claim that possession of all and only its actual elements is an essential property of any given set is a natural one. The present issue, however, is whether this form of reasoning transfers across to the plural case. And here a number of questions arise. Most fundamentally, it is not entirely obvious that pluralities need be regarded as extensional in any sense: the temptation to think this as immediate as it is in the set-theoretic case arises from the truism `a plurality just \textit{is} its members\footnote{It is a nice, and generally ignored question, how precisely the `is' in this sentence should be understood. It looks like it cannot be an identity predicate on pain of violating type restrictions (and thereby introducting the threat of a form of Cantor's paradox).}'. So far, so uncontroversial: the cutlery set represents no addition of being over and above the knife, the fork, and the spoon. Does extensionality follow? Unspoken in the background is the assumption that the world is `carved up' into pluralities prior to, and independent of, our practices of referring to, and quantifying over things plurally (what we might call \textit{plural realism}). Perhaps there are good reasons to think this. But were plural realism denied in favour of the view, say, that a plurality just is some things \textit{as picked out by an expression with a particular sense}\footnote{I take it that the view sketched here would sit comfortably with a development of the conditional direct reference theory developed in (\cite{hewitt:12})}, extensionality would fail. This would, amongst other things, open the door to the denial of the plurality/ group distinction\footnote{A group being a collection of people picked out by some role or feature, such as the Supreme Court justices. See \cite{uzquiano:04}.}, a move which in addition to ontological economy\footnote{Although an alternative view, congruent with recent advocacy of easy ontology would concede the existence of groups, since identity conditions for groups can be supplied in terms of the corresponding pluralities, but deny their explanatory role \cite{thomasson:15}.} has in its favour the undoubted fact that (what philosophers generally take to be) many groups are picked out by noun-phrases linguistically indiscernible from plural NPs. At this point questions would likely be raised about whether commitment to pluralities is any less costly than commitment to objects \cite[]{linneboflorio:15}.\\

Let's grant extensionality for pluralities then, at least for the sake of argument: problems remain with the argument. It is taken for granted that extensionality is captured  by \eqref{Ind}, where modal operators may occur within the substitutends for the sentential metavariables. Should somebody not antecedently disposed to accept \textit{pn} accept this characterisation of plural extensionality? Here is one reason why they ought not to: a non-singularist meaning theory for plural vocabulary could hold, consistently with the rejection of a parallel view for singular terms, that our referential practices with plurals are fine-grained, allowing the pluralities we track with them to come apart modally. After all, they might insist, openess to this kind of proposal is part of what it is to take plurals seriously. And this being so, allowing modal operators to be substituted within \eqref{Ind} is question-begging.\\ 

None of this is decisive either way. Sceptical doubt has been introduced to Linnebo's plausibility argument, which ought perhaps to cause us to question whether we can formulate a logico-metaphysical account of pluralities independent of developing a better understanding of the function of plural expressions in language. We'll return to that point in due course. Perhaps, though, we can short-circuit that imperative for present purposes. Might it not be that formal arguments for \textit{pn} are available, showing that the contested principle follows from commonly accepted (or at least plausible) rules governing plurals? Linnebo thinks so, and supplies three arguments to that effect. To these our attention now turns.

\section{}

Here and in what follows, we will work in the logic \textbf{PFO+} enriched with a singular existence predicate. For details of \textbf{PFO+}, see the appendix. We will allow ourselves the free use of modal operators, and leave the syntax of these intuitive\footnote{Although we'll take it that the background modal logic is normal}. Linnebo's first argument for \textit{pn} we will call \textit{Uniform Adjudication (UA)}. Versions of UA are supplied for \textbf{PFO+} with and without an existence predicate. Since nothing important turns on which argument is at issue, we will focus on that without the existence predicate for the sake of simplicity.\\

Add to the language of \textbf{PFO+} a dyadic plural term forming operator `+', which takes a singular term in its first argument place and a plural term in its second. The intended interpretation of this operator is as denoting \textit{adjunction}, the operation of `adjoining one object to a plurality' \cite[11]{linnebo:16}. The following seems  plausible as a principle governing adjunction:

\begin{equation}
 \tag{UniAdj} \Box\forall x(x\prec xx+a\leftrightarrow x\prec xx\vee x=a)
\end{equation}

Linnebo doesn't provide an argument for (Uniadj), but one may be easily be found. The principle is, after all plausibly understood as (implicitly) definitional of the adjunction operator, and as such analytic and therefore necessary\footnote{Those beset by Quinean scruples about analyticity are invited to reformulate this argument.} Given (UniAdj), the next stage in (UA) is to assume $a\prec xx$. This gives us $xx\equiv xx+a$ by (UniAdj). We are licensed to necessitate this by (Cov), yielding:
\begin{equation}
 \label{Lin1} \Box(xx\equiv xx+a)
\end{equation}
The next stage is to argue that (UniAdj) yields:
\begin{equation}
\label{Lin2}  \Box(a\prec xx+a)
\end{equation}
According to Linnebo (UniAdj) \textit{entails} \eqref{Lin2} within the scope of our assumption, but no proof is provided. Let's talk through how one might go. Instantiating the bound variable with $a$, application of \textbf{K} to the R-L direction of (UniAdj) delivers:
\begin{equation}
 \label{Hew1} \Box(a\prec xx\vee a=a)\rightarrow\Box(xx+a)
\end{equation}
The necessity of identity and \textit{modus ponens} give us \eqref{Lin2}. Returning to (UA), with \eqref{Lin2} under our belts, we can appeal to \eqref{Lin1} and get $\Box(a\prec xx)$. We discharge the assumption and we are done.\\

The problem with (UA) as a case for (NecInc) -- remember that nothing here turns on the presence of an existence predicate -- is the invocation of (Cov) to necessitate $\eqref{Lin1}$. Recall that (Cov) is derived using (Indisc):

\begin{equation}
 \tag{Indisc} xx\equiv yy\rightarrow(\phi(xx)\leftrightarrow\phi(yy)) 
\end{equation}

In the context of an argument for a component of \textit{pn} this looks problematic. In order to get the derivation of (Cov) we need modal vocabulary to be admissible in the substitutends for the sentential metavariables. But it is doubtful than anyone who doubts \textit{pn}, and so anyone with worries about (NecInc) should allow this. For this admission expresses formally the doctrine that pluralities are \textit{extensional}. For sure, as the discussion above noted, this need not be a \textit{hard} extensionality; but then perhaps Linnebo is correct in saying that soft extensionalism is an unstable position. More fundamentally, however, we minuted reasons to doubt that pluralities are extensional in \textit{any} sense. And even granting that, we observed that the admission begs questions about the constraints on the semantic behaviour of plurals in modal contexts. This is a defeating worry for the use of (Cov) in (UA). Someone not antecedently committed to \textit{pn} ought only to admit substitutends in (Indisc) that do not contain modal operators. This is not enough to derive (Cov), so (UA) is not available to the non-\textit{pn} believer. But then, in particular, it is not available when arguing which a non-\textit{pn} believer; which is to say it lacks the required suasive force.

\section{}

The second argument is owing to Williamson, and termed by Linnebo \textit{partial rigidification} (PR) \cite[699-700]{williamson:10}. Considering a fixed domain context first, this proceeds from an instance of plural comprehension:

\begin{equation}
 \exists yy(xx\equiv yy\wedge\forall x(x\prec yy\rightarrow\Box x\prec yy))
\end{equation}
$yy$ here is a \textit{partial rigidification} of $xx$. Now assume $a\prec xx$ and consider $yy$, the partial rigidification of $xx$. This gives us $\Box a\prec yy$. Together with $\Box(xx\equiv yy)$, which we have by (Cov), the derivation of $\Box a\prec  xx$ is immediate. Discharging the assumption gives us (NecInc).\\

Linnebo considers this persuasive, but of limited value, since the fixed domain assumption is unlikely to appeal to anyone not signed up to Williamsonian necessitism\footnote{On which, see \cite{williamson:13}.}. I doubt even that (PR) is of use in arguing for \textit{pn} even \textit{modulo} a fixed domain, because like (UA) it appeals to (Cov), and as we have seen this looks question-begging. Things look unfavourable also when (PR) is adapted for variable domains. To do this, we use a different instance of plural comprehension:

\begin{equation}
 \exists yy\bigl(xx\equiv yy\wedge\forall x[x\prec yy\rightarrow\Box(Eyy\rightarrow\Box x\prec yy)]\bigr)
\end{equation}

We proceed in a similar fashion as for the fixed-domain case, this time concluding that $\Box(Eyy\rightarrow a\prec yy)$ on the assumption that $a\prec xx$. Appealing to (Cov) would get us $\Box(Exx\prec a\prec xx)$, allowing us to discharge and get (NecInc) only given that $\Box(Exx\rightarrow Eyy)$, but as Linnebo himself acknowledges this is not something that somebody who rejects \textit{pn} is likely to allow. If one is prepared to grant that $yy$ are actually all and only $xx$, but contingently so, what stands in the way of granting the possibility that $yy$ exist but $xx$ do not?\\

For this reason Linnebo takes (PR) to have force against the denier of \textit{pn} in a fixed domain, but not in a variable domain context. He is certainly correct regarding the latter case, but it is unclear that the argument is any more convincing when the domain is fixed. As we saw, in that case (PR) still relies on (Cov), and at this point the \textit{pn}-denier ought to refuse to follow. We arrive here at a recurring theme throughout the present paper; formal arguments for \textit{pn} are less successful than even Linnebo, a sober and balanced commentator, allows. The principle does not appear susceptible to proof from premises which are not too close to the principle itself for the proof to have suasive force. This will become more apparent when we consider the third and final of Linnebo's formal arguments.

\section{}
Say that a plurality is traversable just in case its members can be exhaustively listed. This is straightforward in the case of a finite plurality. For example, let $xx$ be a plurality traversed thus:

\begin{equation}
\label{trav} \forall x(x\prec aa\leftrightarrow x=a\vee x=b\vee x=c)
\end{equation}
Linnebo says that we can assert \textit{uniform traversability}, by which is meant the necessitation of wffs such as (\eqref{trav})\footnote{Note that the principle of uniform traversability for finite pluralities is not statable in the object language.}, thus,
\begin{equation}
 \Box\forall x(x\prec xx\leftrightarrow x=a\vee x=b\vee x=c)
\end{equation}

In the context of an argument for \textit{pn} this might look fatally question-begging, a point to which we'll return in due course. First we need to deal with the case of infinite pluralities. Here substantial infinitary resources are required\footnote{The language of $\mathcal{L}_{\infty\omega}$ expands that of first-order logic with identity by admitting disjunctions and conjunctions of any transfinite length, whilst allowing only finite blocks of quantifiers. We assume, moreover, that the language is equipped with a proper class of individual constants.}. Working in plural $\mathcal{L}_{\infty\omega}$, and adopting the notation of \cite{hewitt:12}, uniform traversability in this case may be stated with $i$ ranging over sufficient ordinals to index the members of $xx$,

\begin{equation}
 \Box\bigl(x\prec xx\leftrightarrow \bigvee_{\forall i}(x=a_{i})\bigr) \label{necdis} \tag{NDIS}
\end{equation}

Given \textbf{T} in the background modal logic we can prove\footnote{For details see \cite[]{hewitt:12}.} $\Box y\prec xx$ on the assumption that $y\prec xx$, giving us (NecInc) for fixed domains. A similar argument is forthcoming in when an existence predicate is predicate, giving us (NecInc) \textit{simpliciter}. However, as Linnebo himself admits,

\begin{quote}
 [the] premise of universal traversability is little other than an infinitary restatement of our target claim that a plurality is fixed in its membership as we shift our attention from one possible world to another. \cite[14]{linnebo:16}
\end{quote}

As advertised, there is little reason to believe that an assertion of uniform traversability will do anything to convince anyone not antecedently given to assent to \textit{pn}. This much Linnebo himself concedes. However, the other formal arguments he deploys are in equally bad shape, as we have seen. In each case the complaint is the same: the argument will not persuade somebody not already disposed to assent to \textit{pn}, since at a crucial stage it deploys an argument that such a person has no reason to accept. We know that (NecInc), and therefore \textit{pn}, is independent of any normal modalisation of \textbf{PFO+}, so to the extent that these logics capture pre-formal reasoning about plurals and modality our target doesn't fall out of principles governing thought. What is more \textit{pn} has been questioned by a minority of philosophers. Still, the author's impression is that the great majority assent to it. Given the use to which the principle is put in philosophy, it would be unfortunate if it lacked adequate justification; yet the position we have reached suggests that this might be a live possibility. Can more be said?

\section{}

We might be tempted to think that too much is being conceded to scepticism. After all, a plural variable `$xx$' is supposed to formalise a natural language expression such as `these things'. But, one can imagine an exasperated interlocutor protesting, if \textit{this} thing is one of \textit{these} things then \textit{of course} it is necessarily the case that this is one of \textit{these things}, or else they wouldn't be \textit{these things}. Unfortunately, this is simply an affirmation of faith in place of an argument: the claim that `these things' designates all and only these things with respect to a counterfactual situation is implicit in `or else they wouldn't be these things'. The expression `these things' is being used to justify the claim that of necessity these things are all and only these things. Yet it can only do this on assumption that `these things' is semantically rigid in the sense that\footnote{The term `(semantically) rigid' used of a plural NP is dangerously ambiguous. My usage here has it that $\alpha\alpha$ is rigid iff (a) if $x$ is one of the things designated by $\alpha\alpha$ then $x$ is one of the things designated by $\alpha\alpha$ with respect to any world at which $\alpha\alpha$ designates some things and $x$ exists; and (b) if $x$ is not one of the things designated by $\alpha\alpha$ then $x$ is not one of the things designated by $\alpha\alpha$ with respect to any world at which $\alpha\alpha$ designates some things and $x$ exists. An alternative usage would have it that simply $\alpha\alpha$ is rigid iff $\alpha\alpha$ designates the same plurality with respect to any world at which $\alpha\alpha$ designates any plurality. My view is that latter claim is utterly harmless, indeed uninformative, since we have no grasp of what pluralities \textit{are}, what it is for some things to be considered as many, apart from our use of plural language. Pluralities just are what plural expressions designate.} if it designates \textit{anything} with respect to a counterfactual situation then it designates all and only the things it actually designates. This semantic presupposition is not \textit{pn}, to identify the two would be a gratuitous confusion of use and mention; however it is clearly closely related. As we shall now see.\\

The formal logic of plurals is of interest because it allows us to codify plural concepts we express in natural language and to capture the rules of deduction governing these. Were \textbf{PFO+} merely a mathematical system, understood formalistically, then there would be no barrier to importing any interaction with modal operators we care to mention in order to extend the system. For as long as our concern is mere symbolic games, one game is as good as another. But on the intended interpretation of \textbf{PFO+} as a \textit{plural} logic, the metatheoretic statement which appropriately formalises the claim that $\alpha\alpha$ is semantically rigid just in case if it designates \textit{anything} with respect to a counterfactual situation then it designates all and only the things it actually designates is supposed to model a claim that could be made about natural language plurals. And a meaning theory for a natural language should be statable in (the same) natural language itself, worries about the semantic paradoxes notwithstanding\footnote{I agree with Priest about this much, although dissent from his dialethic response to paradox \cite{priest:87}. Rather I would want to follow Rumfitt in maintaining that paradoxical sentences do not succeed in making statements \cite[]{rumfitt:15}.}. But now it is clear that the move to the metalanguage imposes no rigidity on object language plurals that wasn't already implicit in the statement of object language \textit{pn} itself, for ultimately the semantic principle for the formal language should impose nothing not present in the natural language which is our  object of investigation. In that language: `these things', used in the statement of a theory of meaning for the language, means \textit{these things}. The constraint that meaning be homophonic ensures that the invocation of the metatheory in the fashion proposed by the idignant objector of the previous paragraph gives no new suasive purchase.\\

Where does this leave us with respect to \textit{pn}? If a widely accepted logic of plurals doesn't deliver \textit{pn} as a theorem, no matter which (normal) modal logic it is combined with, and if obvious formal arguments departing from the accentuation of plural logic with metaphysico-semantic principles concerning plurals, what possible motivation could there be for adopting \textit{pn}? A temptation at this point is to appeal to more robustly metaphysical considerations about the nature of pluralities. Provisional on metaphilosophical choices minuted below, this gives the impression of an admission of defeat under the guise of further enquiry. For whatever the nature of pluralities might be (and here I already fear that language is packing its suitcase, if not quite yet on holiday: pluralities are not \textit{entities}, what could it be for them to have a \textit{nature}?), it is surely something that we approach through our canonical means of latching on to them, that is through the use of plural language. How else can I explore what it is to be a plurality other than through attention to our plural talk and reasoning? Do not be misled by comparisons with the scientific study of entities of some particular class. I can explore the nature of water through using an electron microscope: my use, notoriously, will not settle the question. But `plurality' is not a sortal of the same kind as `water': it expresses a general, formal, concept applicable across a wide range of collections of entities of all sorts and categories. In its formality it belongs with `object', rather than `positron'; and the prospects for a substantial enquiry into the nature of objects, as opposed to a logico-semantic making lucid of the place of objects in the structure of the world as approached through language, are not good.\\ 

We cannot step outside of language to study the general structure of the world in its metaphysical purity, whatever that would involve. Such is our logocentric predicament\footnote{For the record, I don't see it as a \textit{much} of a predicament. It seems to be a serious problem for as long as we are caught up in a picture according to which language somehow distorts our access to reality. But the picture is not compulsory. We can, alternatively, view language as our \textit{means} to grasp reality conceptually, to make its contents the objects of reasoning. As Dummett puts it, `language may be a distorting mirror, but it is the only mirror we have' \cite[6]{dummett:14}.}. What is it about the use of plural terms -- `these students', `the Channel Islands', `Pixies' (if we take the last to be plural, and not a singular name for a \textit{group}) -- that licences the application of \textit{pn}? This, it seems to me, is the only question that could give us a good reason to adopt the principle. And here there nothing in use that proponents of the principle have thus far rallied to its support, except with respect to the proper sublplurality of plural terms consisting of compound names -- `Ant and Dec', `2 and 17' -- and these are not the contested cases\footnote{In many cases, the reference of `we' is anaphoric for that of a compound name. Unfortunately this makes the cute example with which the present paper began something of a ladder to be taken away once the dialectic is going.}. Moreover, there is some evidence from usage against \textit{pn}: `You should be careful what you are saying; Smith could have been one of those men', to adapt an example of Dorothy Edgington's\footnote{Cited in \cite[120-1]{rumfitt:05}.}. To complain that here `those men' is associated with descriptive content (`the men talking in that corner at this party') is simply to shift the dialectic, and not in a manner that looks likely to favour \textit{pn}. Are plural terms associated with some kind of intensional content, a sense distinct from their reference? What would decide this issue? A prime datum here seems to be precisely the kind of case Edgington proposes\footnote{Note that we are dealing here with a plural \textit{demonstrative}, so comparisons with singular descriptions such as `the number of planets' to motivate the thought that `those men' is not a genuine term are illicit. Against the objection that the sortal supplies descriptive content we can reply that (a) it doesn't supply \textit{enough} content to secure \textit{pn}, and (b) There is a case to be made that \textit{all} terms, including singular names, are associated with sortals \cite[]{}\cite[]{}, so this line of objection is in danger of generalising uncontrollably.}, and this supports not \textit{pn} but its negation, since it provides a \textit{prima facie} example of an instance for which the universal claim fails. \\

I do not take this to be decisive. Perhaps there is an argument to be had in favour of \textit{pn} that answers this objection. It might be, moreover, that reasonable constraints on a meaning theory motivate acceptance of the contested principle: what happens, for example, if we formalise modal plural logic along natural deduction lines, imposing the familiar requirements of harmony and conservativity on rules for logical vocabulary and treating plural terms correspondingly in an inferentialist fashion? The point is merely that the work of motivating \textit{pn} philosophically remains to be done, and that in the absence of such work the appeal to the principle in support of strong metaphysical theses is illegitimate.\\

The frustrated defender of \textit{pn} has another avenue of response. Isn't to seek to ascertain the truth of the matter as regards \textit{pn}, as we have been doing, in isolation from the project of developing our best overall theory of the world to succumb to methodological error? If the benefits of adding \textit{pn} to one's theoretical arsenal outweigh the costs, for example in terms of sitting uncomfortably with current usage, then one should adopt the principle. This after, all is how science proceeds, and metaphysics and logic are part of total science. Thus, for example, Williamson has advocated a `modal science' and has more generally stressed the methodological continuity between philosophy and science \cite{} \cite[]{}. This position is, to my mind, unconvincing. The disanalogies between metaphysics and the natural sciences are legion, most obviously in the failure of metaphysical theories to make empirically testable intuitions. It might be retorted that metaphysical theories can be tested against \textit{intuitions}, perhaps by means of thought experiments, these providing something analogous to the experiments of natural scientists. The role of intuitions in philosophy has been subject to a good deal of critical scrutiny of late \cite{}\cite{}\cite{} and the analogy between thought experiments and experiments \textit{simpliciter} has been questioned, rightly in my view. However, even within the general methodological framework in question there would seem to be particular concerns about appeals to intuition to support \textit{pn}. For here, if we are not to confine our attention to imagined scenarios of the Edginton type (which hardly offer unambiguous support to \textit{pn}, we will be arguing about intuitions and /or counterfactual judgements concerning either the truth of sentences expressed using English plural terminology (`In this situation, this would be one of these') or philosophical terms of art (`plurality'). In the first case, the above mooted accusation of circularity recurs; in the latter, the supposed role of intuition recedes into the background, and we are left simply with \textit{a priori} metaphysical disputation of the old school, with no obvious similarity to the natural sciences or any other uncontroversially truth-conducive enterprise. Either way, moreover, there is a real danger of dominant views amongst current metaphysicians on \textit{pn} reinforcing themselves by appeal to the intuitions of the already-convinced, a procedure that is of no epistemic value.\\

If what I have suggested above is correct with respect to \textit{pn} then it looks likely to generalise beyond the case of \textit{pn} to other modal intuitions to which common and important appeal is made in philosophy. It is commonplace for \textbf{S5}, or at least \textbf{S4} to be appealed to as the `right' modal logic, and for metaphysical investigation to proceed with one of these as the background modal logic. There are disimilarities with the case of \textit{pn}: notably appeals to model theoretic semantics and to counterfactuals in support of \textbf{S5}. In both cases, I contend that circularity worries similar to those noted about arguments for \textit{pn} arise. Establishing that will have to wait for a subsequent paper. What is clear, however, is that in the absence of these arguments support for strong modal logics will rest on judgements concerning iterated modalities that are at least as precarious as those concerning expressions of \textit{pn} in natural language. Perhaps we should just be more modest in our claims about modality.\\

\appendix
The logic \textbf{PFO+} extends first-order logic with identity. To the lexicon we add:

\begin{itemize}
 \item Denumerably many \textit{plural variables}: `$xx_{1}$', `$xx_{2}$'$\ldots$
 \item Denumerably many \textit{plural constants}: `$aa_{1}$', `$aa_{2}$'$\ldots$
 \item Denumerably many (monadic) \textit{plural predicates}: `$FF_{1}$', `$FF_{2}$'$\ldots$
 \item The logical dyadic predicate `$\prec$', which we read `is one of'.
\end{itemize}
Customary abbreviations and modifications of notation are allowed. We add new formation rules:
\subsection{Atomic wffs}
Where $t$ ranges over singular, and $tt$ plural, terms:
\begin{itemize}
 \item $\ulcorner t\prec tt\urcorner$ is a wff.
\end{itemize}
\subsection{Molecular wffs}
Where $\phi$ is a wff, $vv$ a plural variable, and with the usual constraints to avoid clash of variables:
\begin{itemize}
 \item $\ulcorner\exists vv\:\phi\urcorner$ is a wff.
 \item $\ulcorner\forall vv\:\phi\urcorner$ is a wff.
\end{itemize}
Nothing other than elements of the closure of the atomic wffs under the formation rules is a wff.
\section{Proof system}
To any proof system for first-order logic with identity we add the axiom:\\

\noindent\textbf{Nonemptiness}: $\forall xx\exists y\: y\prec xx$\\

And the axiom schema:\\

\noindent \textbf{Comprehension}: $\exists x\: \phi\rightarrow \exists xx\forall x\: (x\prec xx\leftrightarrow \phi)$\\

We can also introduce the defined dyadic plural predicate `$\equiv$', which we read `are the same things as':\\

\noindent\textbf{Extensionality}: $\forall xx\forall yy\: xx\equiv yy\leftrightarrow \forall x(x\prec xx\leftrightarrow x\prec yy)$\\

\section{Model-theoretic semantics}

Following \cite{hewitt:12} we will develop a model theoretic semantics for the language of \textbf{PFO+}. We will restrict our attention to fixed domains and show that, even with this restriction, (NecInc) is not valid. We will then add an existence predicate and consider variable domains. We will see that this does not affect the result.\\

Let a frame $\mathcal{F}=(S,R)$, with $S$ non-empty and $R$ a relation on $S$. A model $\mathcal{M}$ on $\mathcal{F}$ is  $\mathcal{M}=(S,R,D,I)$, with $D$ non-empty and $I$ a function that makes assigments to the non-logical vocabulary:
\begin{itemize}
 \item To each \emph{singular constant} some $d\in D$.\\
\item To each \emph{n-adic predicate}, for each $s\in S$, some $d\subseteq D^{n}$.\\
\item To each \emph{plural constant}, a non-empty set  $(x,y) x,y\in P$\ s.t. $\forall (x,y)\in P\: x\in S\wedge y\subseteq D\wedge \forall z\: (<x,z)\in P\rightarrow z=y)$\\
\item To each \emph{plural predicate}, for each $s\in S$ some $p\subseteq \wp(D)$.
\end{itemize}

Given a model, we define a valuation function $v$, which assigns to each singular variable $x$, $v(x)\in D$ and to each plural variable $xx$, $v(xx)\subseteq S\times(D-\emptyset)$, s.t. $\forall s\in S\: ((s,a)\in v(xx)\wedge (s,b)\in v(xx))\rightarrow a=b$.\\

Satisfaction can be defined. Writing $\mathcal{M}, s\Vdash_{v} \phi$ to indicate that $v$ and $\mathcal{M}$ satisfy $\phi$ at $s$ we proceed as follows, with $t$ ranging over singular terms.

\begin{enumerate}
 \item For an n-adic predicate `$F$', $\mathcal{M}, s\Vdash_{v} Ft_{1}\ldots t_{n}$ iff $(v(t_{1}\ldots v(t_{n}))\in I(F,s)$.
\item For a plural constant `$aa$', $\mathcal{M}, s\Vdash_{v} t\prec aa$ iff .$\exists (m,n)\in I(aa)\: (m=s\wedge t\in n)$
\item For a plural variable `$xx$', $\mathcal{M}, s\Vdash_{v} t\prec xx$ iff $\exists (m,n)\in v(xx)\: (m=s\wedge t\in n)$.
\item For a plural predicate `$FF$' and a plural constant `$aa$', $\mathcal{M}, s\Vdash_{v} FFaa$ iff $\exists (m,n)\in I(aa)\: (m=s\wedge n\in I(FF,s)$.
\item For a plural predicate `$FF$' and a plural variable `$xx$', $\mathcal{M}, s\Vdash_{v} FFaa$ iff $\exists (m,n)\in v(xx)\: (m=s\wedge n\in I(FF,s)$.

\end{enumerate}

Compound wffs are dealt with recursively:

\begin{enumerate}
 \item $\mathcal{M}, s\Vdash_{v} \neg\phi\text{ iff }\mathcal{M},s\nVDash_{v}\phi$.
\item $\mathcal{M}, s\Vdash_{v} (\phi\wedge\psi)\text{ iff }\mathcal{M},s\Vdash_{v}\phi\text{ and }\mathcal{M},s\Vdash_{v}\psi$.
\item $\mathcal{M}, s\Vdash_{v} \Box\phi\text{ iff }\forall u\in S\text{ if }sRu\text{ then }\mathcal{M}, u\Vdash_{v} \phi$.
\item $\mathcal{M}, s\Vdash_{v} \forall x\: \phi\text{ iff for every valuation }w\text{, which differs from }v$\\ $\text{ at most with respect to the assignment to }`x', \mathcal{M}, s\Vdash_{w} \phi$.

\item $\mathcal{M}, s\Vdash_{v} \forall xx\: \phi\text{ iff for every valuation }w\text{, which differs from }v$\\ $\text{ at most with respect to the assignment to }`xx', \mathcal{M}, s\Vdash_{w} \phi$

\end{enumerate}

A wff is \textit{true} in $\mathcal{M}$ if it is satisfied on every assignment at every world in $\mathcal{M}$. A wff is \textit{valid} relative to a class of frames, if it is true in every model based on a frame from the class. A wff is valid simpliciter if it is true in every model. The concepts of truth and validity can be extended from wffs to theories in the natural way. Consequence is similarly defined naturally.\\

\begin{theorem}
 (NecInc) is not \textbf{K}-valid.
\end{theorem}

\begin{proof}
 Let $\mathcal{F}=(\{ 0,1\}, R)$. Now consider a model $\mathcal{M}$. Let $R=\{(0,1),(1,0)\}$ and let $D=\{\pi, e\}$. Consider  a valuation $v$  including $v(x)=\pi$ and $v(xx)=\{(0, \{\pi\}), (1,\{ e\})\}$. At 0, `$x\prec xx$' is satisfied, but `$\Box x\prec xx$' is not, since $0R1$ and $\mathcal{M}, 1\nVdash_{v} x\prec xx$. Hence (NecInc) is not satisfied at 0. It follows that (NecInc) is is invalid in $\mathcal{M}$, and thus that it is \textbf{K}-invalid.
\end{proof}

The soundness of the semantics for \textbf{PFO+} is proved by a routine induction on the length of proofs and we omit the details.

\section{Adding an existence predicate}

To work with \textbf{PFO+} in a variable domain context, we first of all add a special predicate `$E$' to the lexicon s.t. $\ulcorner Ett\urcorner$
and $\ulcorner Et\urcorner$ are wffs\footnote{We can think of `$E$' in terms of two predicates, orthographically indiscernible, one singular and one plural. Alternatively, see footnote \textit{v} above.} We modify the axioms as follows:

\noindent\textbf{V-Nonemptiness}: $\forall xx\: (Exx\rightarrow \exists y\: (Ey\wedge y\prec xx))$\\

\noindent \textbf{Comprehension}: $\exists x\: (Ex\wedge\phi)\rightarrow \exists xx\: (Exx\wedge\forall x\: (x\prec xx\leftrightarrow(Ex\wedge\phi)))$\\

We modify the model semantics so that $\mathcal{M}=(S,R,D,I,N)$. $N$ assigns to each world $s$ $N_{s}\subseteq D$, the \textit{inner domain} of $s$. We interpret `$E$' s.t. `$Et$' is satisfied at a world under a valuation iff $v(t)\in N_{s}$; similarly for plural terms iff $v(tt)\subseteq N_{s}$. We can see that (NecInc) still fails to be valid: simply modify the above countermodel such that the inner domain is coextensive with the domain at each world.

 {*}\bibliography{bibliography}

\begin{thebibliography}{}

\bibitem[Blackburn et~al., 2002]{blackburn:02}
Blackburn, P., de~Rijke, M., and Venema, Y. (2002).
\newblock {\em Modal {L}ogic}.
\newblock Cambridge {U}niversity {P}ress, Cambridge.
\newblock Cambridge {T}racts in {T}heoretical {C}omputer {S}cience 53.

\bibitem[Boolos, 1998a]{boolos:98f}
Boolos, G. (1998a).
\newblock The iterative conception of set.
\newblock In \cite{boolos:98}, pages 13--29.

\bibitem[Boolos, 1998b]{boolos:98}
Boolos, G. (1998b).
\newblock {\em Logic, Logic and Logic}.
\newblock Harvard University Press, Cambridge, MA.

\bibitem[Boolos, 1998c]{boolos:98b}
Boolos, G. (1998c).
\newblock To be is to be a value of a variable (or to be some values of some
  variables).
\newblock In \cite{boolos:98}, pages 54--72.

\bibitem[Dummett, 2014]{dummett:14}
Dummett, M. (2014).
\newblock {\em Origins of {A}nalytical {P}hilosophy}.
\newblock Bloomsbury, London.

\bibitem[Florio and Linnebo, 2015]{linneboflorio:15}
Florio, S. and Linnebo, O. (2015).
\newblock On the {I}nnocence and {D}eterminacy of {P}lural {Q}uantification.
\newblock {\em Nous}.
\newblock Forthcoming.

\bibitem[Hellman, 1989]{hellman:89}
Hellman, G. (1989).
\newblock {\em Mathematics {W}ithout {N}umbers}.
\newblock Oxford University Press, Oxford.

\bibitem[Hewitt, 2012]{hewitt:12}
Hewitt, S. (2012).
\newblock Modalising plurals.
\newblock {\em Journal of {P}hilosophical {L}ogic}, 41:853--875.

\bibitem[Linnebo, 2012]{linnebo:12}
Linnebo, O. (2012).
\newblock Plural quantification.
\newblock Article in the Stanford Encyclopedia of Philosophy.
\newblock Available online at http://plato.stanford.edu/entries/plural-quant/ .

\bibitem[Linnebo, 2013]{linnebo:13}
Linnebo, O. (2013).
\newblock The {P}otential {H}ierarchy of {S}ets.
\newblock {\em Review of {S}ymbolic {L}ogic}, 6(2):205--29.

\bibitem[Linnebo, 2016]{linnebo:16}
Linnebo, O. (2016).
\newblock Plurals and {M}odals.
\newblock {\em Canadian {J}ournal of {P}hilosophy}.
\newblock DOI: 10.180/00455091.2015.1132975.

\bibitem[Marcus, 1961]{marcus:61}
Marcus, R.~B. (1961).
\newblock Modalities and intensional languages.
\newblock {\em Synthese}, 13(4):303--322.

\bibitem[Priest, 1987]{priest:87}
Priest, G. (1987).
\newblock {\em In {C}ontradiction}.
\newblock Martinus Nojhoff, Dordrecht, 1st edition.

\bibitem[{P}riest, 2015]{priest:15}
{P}riest, G. (2015).
\newblock Is the {T}ernary {R} {D}epraved?
\newblock In {C}aret, C.~R. and {H}jortland, O.~T., editors, {\em Foundations
  of {L}ogical {C}onsequence}, pages 121--135. Oxford {U}niversity {P}ress,
  Oxford.

\bibitem[Rayo, 2007]{rayo:07}
Rayo, A. (2007).
\newblock Plurals.
\newblock {\em Philosophical {C}ompass}, 2(3):411--27.

\bibitem[Rumfitt, 2005]{rumfitt:05}
Rumfitt, I. (2005).
\newblock Plural terms: Another variety of reference?
\newblock In Bermudez, J., editor, {\em Thought, Reference, and Experience :
  Themes from the Philosophy of Gareth Evans}, pages 84--123. Clarendon Press,
  Oxford.

\bibitem[Rumfitt, 2015]{rumfitt:15}
Rumfitt, I. (2015).
\newblock {\em The {B}oundary {S}tones of {T}hought : an essay on the
  philosophy of logic}.
\newblock Oxford {U}niversity {P}ress, Oxford.

\bibitem[{T}homasson, 2015]{thomasson:15}
{T}homasson, A.~L. (2015).
\newblock {\em Ontology {M}ade {E}asy}.
\newblock Oxford {U}niversity {P}ress, Oxford.

\bibitem[Uzquiano, 2004]{uzquiano:04}
Uzquiano, G. (2004).
\newblock The {S}upreme {C}ourt and the {S}upreme {C}ourt {J}ustices : {A}
  {M}etaphysical {P}uzzle.
\newblock {\em Nous}, 38(1):135--153.

\bibitem[Williamson, 2010]{williamson:10}
Williamson, T. (2010).
\newblock Necessitism, contingentism and plural quantification.
\newblock {\em Mind}, 119(475):657--748.

\bibitem[Williamson, 2013]{williamson:13}
Williamson, T. (2013).
\newblock {\em Modal {L}ogic as {M}etaphysics}.
\newblock Oxford University Press, Oxford.

\end{thebibliography}
 {*}\bibliographystyle{apalike} 
\end{document}